\documentclass[12pt]{article}
\usepackage[utf8]{inputenc}
\usepackage{amsmath}
\usepackage{amssymb}
\usepackage{tikz}
\usepackage{physics}
\usepackage{graphicx}
\usepackage{array}
\usepackage[margin=1in]{geometry}
\usepackage{ragged2e}
\usepackage[shortlabels]{enumitem}
\usepackage{dsfont}
\usepackage{mathtools}
\usepackage[backend=bibtex]{biblatex}
\usepackage{appendix}
\usepackage{hyperref}
\usepackage{xurl}
\hypersetup{
           breaklinks=true
        }
\def\S{\mathfrak{S}}
\def\inv{^{-1}}
\def\a{\alpha}
\def\ph{\varphi}
\def\s{\sigma}
\def\l{\lambda}
\def\t{\tau}
\def\ra{\rightarrow}

\def\id{\mathrm{id}}
\def\nD{\overline{D}}
\def\nE{\overline{E}}
\everymath{\displaystyle}
\newcommand{\al}[1]{\begin{align*} #1 \end{align*}}

\usepackage{lipsum}
\usepackage{amsthm}
\usepackage{amsmath}
\newtheorem{thm}{Theorem}
\newtheorem{lem}{Lemma}

\theoremstyle{definition}

\newtheorem*{examples}{Examples}
\newcommand{\remove}[1]{}

\newtheorem*{L1}{Lemma~\ref{lem:pimap}}
\newtheorem*{T1}{Theorem~\ref{thm:desc}}
\newtheorem*{T2}{Theorem~\ref{thm:invol}}
\newtheorem*{T3}{Theorem~\ref{thm:subtr}}

\begin{document}
\title{Naturally emerging maps for derangements and nonderangements}
\author{Melanie Ferreri}
\maketitle
\begin{abstract}
    A derangement is a permutation with no fixed point, and a nonderangement is a permutation with at least one fixed point. There is a one-term recurrence for the number of derangements of $n$ elements, and we describe a bijective proof of this recurrence which can be found using a recursive map. We then show the combinatorial interpretation of this bijection and how it compares with other known bijections, and show how this gives an involution on $\mathfrak{S}_n$. Nonderangements satisfy a similar recurrence. We convert the bijective proof of the one-term identity for derangements into a bijective proof of the one-term identity for nonderangements.
\end{abstract}

\section{Introduction}
A \textit{derangement} is a permutation $\s \in \S_n$ such that for all $i \in [n]$, $$\s(i) \neq i,$$ i.e. a permutation which does not fix any element. 
We denote by $D_n$ the set of derangements on $n$ elements, and let $d_n = |D_n|$. Let $E_n$ be the set of permutations of $n$ elements with exactly one fixed point, and let $e_n = |E_n|$.
Two well-known recurrence relations for counting derangements are 
\begin{equation}
\label{eqn:twoterm}
    d_n = (n-1)d_{n-1} + (n-1)d_{n-2}
\end{equation}
 and
 
 \begin{equation}
\label{eqn:oneterm}
    d_n = nd_{n-1} + (-1)^n
\end{equation}
with $d_0 = 1$ and $d_1 = 0$.  

In Chapter 2.2 of \cite{stanley2012enumerative}, Stanley states that it is easy to give a combinatorial proof for (\ref{eqn:twoterm}), but it takes considerably more work to show the one-term recurrence (\ref{eqn:oneterm}) combinatorially. There are several bijective proofs in the literature. 
In \cite{REMMEL1983371}, Remmel proves the one-term identity, along with a $q$-analog of the equation. 
D\'{e}sarm\'{e}nien and Wilf also give bijective proofs in \cite{desar} and \cite{wilf}.
More recently, Benjamin and Ornstein \cite{bo} gave a bijection in four cases, and Elizalde \cite{elizalde2020simple} gave a bijection which, similarly to Remmel's bijection, involves two cases that depend on the disjoint cycle decomposition of $\s$. 


Our objective is to show that a bijective proof of (\ref{eqn:oneterm}) 
emerges naturally from the two-term map and the inductive proof of the one-term identity, both of which will be recalled in Section~\ref{recurrences}.
In order to do this, we show that a bijection proving the recurrence $d_n = (n-1)d_{n-1} + e_{n-1}$ can be applied recursively to obtain a bijection demonstrating the identity $d_n = e_n + (-1)^n$, which has a direct combinatorial interpretation. From there, composing with a map from $E_n$ to $[n] \times D_{n-1}$ yields the desired bijection proving (\ref{eqn:oneterm}).

Furthermore, the map we obtain can be modified to take in elements of both $D_n$ and $E_n$, yielding an involution which exchanges derangements and permutations with exactly one fixed point, excluding one element. Extending this to all elements of $\S_n$, we obtain an involution on the entire symmetric group. 

The recurrence relations for derangements can be used to obtain recurrence relations for \textit{nonderangements} as well, which are permutations with at least one fixed point. We denote by $\overline{D}_n$ the set of nonderangements in $\S_n$. Similarly, we let $\overline{E}_n$ be the set of permutations in $\S_n$ which fix either zero elements or at least two elements. In Section~\ref{nonderangements}, we discuss these sets in more detail, and show how the previously obtained maps can also be used to define a bijection from $\overline{D}_n$ to $\overline{E}_n$, and another bijection from $\overline{E}_n$ to $[n] \times  \overline{D}_{n-1}$. This yields bijective proofs for recurrence relations of nonderangements.


\section{Recurrence relations for $d_n$ and $e_n$} \label{recurrences}


The number of permutations of $[n]$ with exactly one fixed point satisfies the identity
\begin{equation}
    \label{eqn:en}
    e_n = nd_{n-1},
\end{equation}
since a permutation on $[n]$ with exactly one fixed point can be obtained by choosing one of the $n$ elements to be fixed, and then deranging the remaining $n-1$ elements. It will be useful later to have an explicit bijection showing the identity (\ref{eqn:en}), which we define as
$f_n: [n] \times D_{n-1} \ra E_n$.
Given a pair $(m, \s)$, $f_n$ gives a permutation with exactly one fixed point by replacing $m$ with $n$ in the disjoint cycle decomposition of $\sigma$ and fixing $m$ if $m<n$, and otherwise just appending the one-cycle $(n)$. When we have a sequence of cycles that are not disjoint, we will use the convention of composing cycles from left to right. 
To define the inverse map, we make use of some notation provided in \cite{elizalde2020simple}: Given a permutation $\sigma \in S_n$ and $a \in [n]$, we denote by $\sigma \setminus a$ the permutation given by removing $a$ from the disjoint cycle decomposition of $\sigma$. Then $f_n\inv(\tau)$ takes the fixed point of $\tau$ for the first coordinate, and for the second coordinate, it swaps the fixed point of $\tau$ with $n$ and then removes $(n)$ to get a permutation in $D_{n-1}$. So, equation (\ref{eqn:twoterm}) may be rewritten as
\begin{equation}
\label{eqn:new}
   d_n = (n-1)d_{n-1} + e_{n-1},
\end{equation}
with $d_0 = 1$, $d_1 = 0$, and $e_n = nd_{n-1}$.
It is straightforward to show this bijectively, and to construct the recursive bijection in the next section, we will need notation for the bijection showing (\ref{eqn:new}). 
Let 
\al{
\ph_n: D_n &\ra ([n-1] \times D_{n-1} ) \cup E_{n-1} \\
\sigma &\mapsto 
\begin{cases}
(\sigma(n), \sigma \setminus n ) &\text{if } \sigma \setminus n \in D_{n-1}\\
\sigma \setminus n &\text{if } \sigma \setminus n \in E_{n-1}.
\end{cases}
}
Removing $n$ from the cycle factorization of a derangement yields a permutation in $E_{n-1}$ exactly when $n$ was in a transposition in $\sigma$, so the first case occurs exactly when $n$ is not in a transposition in $\sigma$. 
The inverse map $\ph_n\inv$ is as follows:
\al{
\ph_n\inv: ([n-1] \times D_{n-1}) \cup E_{n-1} &\ra D_n \\
(m,\sigma) &\mapsto \sigma (nm) \text{ \,\, if $(m,\s) \in [n-1] \times D_{n-1}$}\\
\tau &\mapsto (an) \tau \text{ \hspace{3mm} if $\t \in E_{n-1}$}
}
where $a$ is the unique fixed point of $\tau$ if $\tau \in E_{n-1}$. In the first case, we essentially insert $n$ in the cycle of $\sigma$ containing $m$, just before $m$, which undoes the operation of $\ph_n$ to get a derangement once again. In the second case, we are just taking the fixed point $a$ and putting it in a transposition with $n$, which undoes the second case of $\ph_n$.

It is know that the recurrence (\ref{eqn:oneterm})
can be proven from (\ref{eqn:twoterm}) by induction. 
Letting $d_0 = 1$, we obtain that $d_1 = 1 + (-1) = 0$, which is indeed the number of derangements on $[1]$. Let $n>1$, and suppose for induction that the result holds for $d_{n-1}$. Then, using the relation (\ref{eqn:twoterm}), we have
\al{
d_n &= (n-1)d_{n-1} + (n-1)d_{n-2} \\
&= (n-1)d_{n-1} + (d_{n-1} - (-1)^{n-1}) \tag{by the inductive hypothesis} \\
&= nd_{n-1} - (-1)^{n-1}\\
&= nd_{n-1} + (-1)^{n}.
}
Then, substituting (\ref{eqn:en}) into (\ref{eqn:oneterm}), we obtain
\begin{equation}
    \label{eqn:onetermnew}
    d_n = e_n + (-1)^n.
\end{equation}
From here, we will use the maps $f_n$ and $\ph_n$ to obtain a bijection demonstrating the relation (\ref{eqn:onetermnew}). 



\section{Defining maps}\label{definingmaps}

Given sets $A$ and $B$, and functions $f: A \ra C$ and $g: B \ra D$, we define 
\al{
f \oplus g: A \sqcup B &\ra C \cup D \\
x &\mapsto \begin{cases}
f(x) &\text{if } x\in A\\
g(x) &\text{if } x \in B
\end{cases}
}
with $\sqcup$ indicating a disjoint union.
Define the permutation $\pi_n \in \S_n$ to be
$$
\pi_n = \begin{cases}
(1\,\, 2)(3\,\,4) \cdots (n-1 \,\, n) &\text{ if } n \text{ is even,}\\
(1\,\, 2)(3\,\,4) \cdots (n-2 \,\, n-1)(n) &\text{ if } n \text{ is odd.}
\end{cases}
$$
Let $\Pi_n$ be the singleton set containing the permutation $\pi_n$, 
and let $\ell_n: \Pi_n \ra \Pi_{n-1}$ be the map between these singleton sets.
Also define $g_n: [n] \times D_{n-1} \ra [n-1] \times D_{n-1} \cup D_{n-1}$ to be the map which removes the first coordinate if it is $n$, and otherwise does nothing. 

To obtain a combinatorial proof of the identity $d_n = e_n + (-1)^n$, we will construct a bijection $\a_n$, which will have an extra element in either its domain or codomain, depending on the parity of $n$.
Denote by $()$ the empty permutation from $\emptyset$ to $\emptyset$. From the definition of $\pi_n$ we have that $\pi_0 = ()$ and $\pi_1 = (1)$. Then the bases cases $\a_0: D_0 \ra E_0 \cup \Pi_0 $ and $\a_1: D_1 \cup \Pi_1 \ra E_1$ are both the identity on sets of size one.

For $n>1$, we define the maps $\a_n$ and $\a_n\inv$ recursively. If $n$ is even, define $\a_n: D_n \ra E_n \cup \Pi_n$ as follows:
\begin{align*}
\alpha_n: D_n \xrightarrow{\ph_n} ([n-1] \times D_{n-1}) \cup E_{n-1} \xrightarrow{\id_{[n-1] \times D_{n-1}} \oplus \a_{n-1}\inv} ([n-1] \times D_{n-1}) \cup D_{n-1} \cup \Pi_{n-1} \\
\xrightarrow{g_n\inv \oplus \ell_n\inv} ([n] \times D_{n-1}) \cup \Pi_{n} \xrightarrow{f_n \oplus \id_{\Pi_n}} E_n \cup \Pi_n
\end{align*}
If $n$ is odd, we define $\a_n: D_n \cup \Pi_n \ra E_n $ as follows:
\al{
\a_n: D_n \cup \Pi_n &\xrightarrow{\ph_n \oplus \ell_n} ([n-1] \times D_{n-1} ) \cup  E_{n-1} \cup \Pi_{n-1} \\
&\xrightarrow{\id_{[n-1] \times D_{n-1}} \oplus \a_{n-1}\inv}( [n-1] \times D_{n-1} )\cup D_{n-1}  \xrightarrow[]{g_n\inv} [n] \times D_{n-1} \xrightarrow[]{f_n} E_n.}
We also define the inverse maps. We let $\a_0\inv: E_0 \cup \Pi_0 \ra D_0 $ send $\pi_0 \mapsto ()$, and let $\a_1\inv: E_1 \ra D_1 \cup \Pi_1$ send $(1) \mapsto \pi_1$. So again, $\a_0\inv$ and $\a_1\inv$ are both the identity on sets of size one. Otherwise let $n>1$.

For $n$ even, define $\a_n\inv: E_n \cup \Pi_n \ra D_n$ as follows:
\al{
\a_n\inv: E_n \cup \Pi_n \xrightarrow{f_n\inv \oplus \id_{\Pi_n}} ([n] \times D_{n-1}) \cup \Pi_n \xrightarrow[]{g_n \oplus \ell_n}
([n-1] \times D_{n-1}) \cup D_{n-1} \cup \Pi_{n-1}\\ \xrightarrow{\id_{[n-1] \times D_{n-1}}\oplus \a_{n-1}} 
([n-1] \times D_{n-1} ) \cup  E_{n-1} \xrightarrow{\ph_n\inv} D_n.
}

For $n$ odd, define $\a_n\inv: E_n\ra D_n  \cup \Pi_n $ as follows:
\al{
\a_n\inv: E_n  &\xrightarrow{f_n\inv} [n] \times D_{n-1}\xrightarrow[]{g_n}
([n-1] \times D_{n-1}) \cup D_{n-1} \\ &\xrightarrow{\id_{[n-1] \times D_{n-1}}\oplus \a_{n-1}}
([n-1] \times D_{n-1} ) \cup  E_{n-1} \cup \Pi_{n-1} \xrightarrow{\ph_n\inv \oplus \ell_n\inv} D_n \cup \Pi_{n}.
}

Note that by construction of $\a_n$, for $n$ even, the composition of $\a_n$ with the map $f_n\inv \oplus \id_{\Pi_n}$ has image $[n] \times D_{n-1} \cup \Pi_n$, and for $n$ odd, the composition of $\a_n$ with the map $f_n\inv$ has image $[n] \times D_{n-1}$. Let 
$$A_n = \begin{cases}
    (f_n\inv \oplus \id_{\Pi_n}) \circ \a_n &\text{if } n \text{ is even,}\\
    f_n\inv \circ \a_n &\text{if } n \text{ is odd,}
\end{cases}
$$
with composition being read from right to left here. Then the composition $A_n$ is a bijection showing the identity (\ref{eqn:oneterm}).


\subsection{Combinatorial description of $\a_n$} \label{sec:description}

After defining $\a_n$, we can trace through the recursion to obtain a direct description of the image of a derangement $\s$. First, $\a_n$ sends $\pi_n$ to $\pi_n$ always. Let $\sigma \in D_n$, $\s \neq \pi_n$. To find $\a_n(\s) \in E_n$, we look at the disjoint cycle decomposition of $\s$, and find the smallest $j$ such that
$$
\sigma = (\cdots) (j \,\,\, j+1) \cdots (n-1 \,\,\, n)
$$
where the initial $(\cdots)$ is any combination of cycles. If there is no trailing pattern of simple transpositions like this, we let $j= n+1$.

Case 1. If $j-1$ is in a 2-cycle, we have 
$$
\sigma = (\cdots)(\cdots j-2 \,\,\, b \cdots ) (j-1 \,\,\, a) (j \,\,\, j+1) \cdots (n-1\,\,\, n)
$$
which is sent to
$$
\a_n(\sigma) = (\cdots)(\cdots j -2 \,\,\, a \,\,\, b \cdots ) (j-1 \,\,\, j) (j+1 \,\,\, j+2) \cdots (n-2\,\,\, n-1)(n)
$$
where any values above $n$ are excluded from the disjoint cycle notation. In particular, if $j-1 = n$, then we have $n$ fixed. 

Case 2. If $j-1$ is not in a 2-cycle, we have
$$
\sigma = (\cdots) (\cdots j-1 \,\,\, a \,\,\, b \cdots ) (j \,\,\, j+1) \cdots (n-1\,\,\, n)
$$
which is sent to
$$
\a_n(\sigma) = (\cdots)(\cdots j-1 \,\,\, b \cdots ) (j \,\,\, a) (j+1 \,\,\, j+2) \cdots (n-2\,\,\, n-1)(n),
$$
where again any values above $n$ are excluded from the cycles. In particular, if $j-1 = n$, then in this case we have $a$ fixed. 

The inverse map operates similarly to $\a_n$. It is described directly in Section \ref{involution}.\\




\begin{examples}[$n=7$]\hfill
\begin{itemize}
    \item $(12)(346)(57)$: Here, there is no pattern of simple transpositions at the end. So we let $j=8$, so $j-1 = 7$, which appears in a 2-cycle. So this is in Case 1. 

    So we take the image of $7$ and put it in the cycle containing 6, directly after 6.

    \begin{center}
        \begin{tikzpicture}[scale=.3]
\draw
(0,0) node (){(}
(1,0) node(){1}
(2,0) node(){2}
(3,0) node(){)}
(4,0) node(){(}
(5,0) node(){3}
(6,0) node(){4}
(7,0) node(){6}
(7.3,0) node (b) {}
(8,0) node(e){)}
(9,0) node(f){(}
(10,0) node (a) {5}
(11,0) node(){7}
(12,0) node(g){)}
;

\path[->,>=stealth]
(a) edge[bend right=70] node [left] {} (b);

\end{tikzpicture}
    \end{center}
    
    Then everything to the right of that cycle becomes $(j-1 \,\, j) \cdots (n) = (7 \,\, 8) \cdots (n)$, but we remove everything above $n=7$, so this just becomes $(7)$. So the image of this derangement is 
    $$
    (12)(3465)(7).
    $$

       \item $(124)(35)({67})$: Here, there is a pattern of simple transpositions at the end, so $j = 6$ and $j-1 = 5$, which appears in a 2-cycle. So this is again in Case 1. 

    We take the image of $5$ and put it in the cycle containing 4, directly after 4.

    \begin{center}
        \begin{tikzpicture}[scale=.3]
\draw
(0,0) node (){(}
(1,0) node(){1}
(2,0) node(){2}
(3,0) node(){4}
(3.3,0) node (b) {}
(4,0) node(){)}
(5,0) node(){(}
(6,0) node(a){3}
(7,0) node(){5}
(8,0) node(e){)}
(9,0) node(f){(}
(10,0) node () {6}
(11,0) node(){7}
(12,0) node(g){)}
;

\path[->,>=stealth]
(a) edge[bend right=70] node [left] {} (b);

\end{tikzpicture}
    \end{center}
    Then everything to the right of that cycle becomes $(j-1 \,\, j) \cdots (n) = (5 \,\, 6)(7\,\, 8) \cdots (n)$, but we remove everything above $n=7$, so this just becomes $(56)(7)$. So the image of this derangement is 
    $$
    (1243)(56)(7).
    $$

   \item $(12)(34)(567)$: Here, there is no pattern of simple transpositions at the end so $j = 8$ and $j-1 = 7$, which appears in a cycle of length $\geq 3$. So this is in Case 2. 

    The image of $7$ gets moved into a 2-cycle with $j =8$:

    \begin{center}
        \begin{tikzpicture}[scale=.3]
\draw
(0,0) node (){(}
(1,0) node(){1}
(2,0) node(){2}
(3,0) node(){)}
(4,0) node(){(}
(5,0) node(){3}
(6,0) node(){4}
(7,0) node(e){)}
(8,0) node(f){(}
(9,0) node(a){5}
(10,0) node () {6}
(11,0) node(){7}
(12,0) node(g){)}
(13.5,0) node (b) {}
;
\draw
(13,0) node(h) {(}
(14,0) node() {8}
(15,0) node(i) {)}
;

\path[->,>=stealth]
(a) edge[bend left=70] node [right] {} (b);

\end{tikzpicture}
    \end{center}
    So we have $(12)(34)(67)(58)$, but we remove everything above $n=7$, so this just becomes 
    $$
    (12)(34)(67)(5).
    $$

    \item $(12345)({67})$: Here, there is a pattern of simple transpositions at the end. We have $j = 6$ and $j-1 = 5$, which appears in a cycle of length $\geq 3$. So this is again in Case 2. 

    The image of $5$ gets moved into a 2-cycle with $j =6$:

    \begin{center}
        \begin{tikzpicture}[scale=.3]
\draw
(0,0) node (){(}
(1,0) node(a){1}
(2,0) node(){2}
(3,0) node(){3}
(4,0) node(){4}
(5,0) node(){5}
(6,0) node(){)}
(7,0) node(e){(}
(7.5,0) node(b) {}
(8,0) node(){6}
(9,0) node(){7}
(10,0) node (f) {)}
;

\path[->,>=stealth]
(a) edge[bend left=70] node [right] {} (b);

\end{tikzpicture}
    \end{center}
    Then everything in the 2-cycles above $j$ changes to $(j+1 \,\, j+2) \cdots (n) = (7 \,\, 8) \cdots (n)$. 
    So we have $(2345)(16)(78)\cdots (n)$, but we remove everything above $n=7$, so this just becomes 
    $$
    (16)(2345)(7).
    $$
\end{itemize}
    \end{examples}
    
We can check that this description indeed matches the effect of applying $\a_n$ by tracing through the recursive definitions of the maps.  Proofs can be found in Appendix \ref{proofs}.


\begin{lem}\label{lem:pimap}
Let $\pi_n \in \S_n$ be as defined earlier.  For all $n \geq 1$, $\a_n$ and $\a_n\inv$ fix $\pi_n$.
\end{lem}


\begin{thm}\label{thm:desc}
The combinatorial description of $\a_n$ matches the recursive definition of $\a_n$.
\end{thm}


 Note that the special case for sending $\pi_n$ to itself does not overlap with the other cases. Suppose $\s \in D_n$, $\s \neq \pi_n$. In Case 1 of the combinatorial description, the resulting permutation always has a cycle of length at least 3. In Case 2, the resulting permutation has a transposition $(j \,\, a)$, where $a = \sigma(j-1) < j-1$, so $(j\,\, a)$ is not a simple transposition. Thus neither of the two cases can yield a permutation which has only simple transpositions in its disjoint cycle decomposition. 



\subsection{Comparing with similar maps}

The bijection presented in \cite{elizalde2020simple} sends derangements in $D_n$ to permutations in $E_n$, the set of permutations with exactly one fixed point, via the following map $\psi$:\\

Let $\sigma \in D_n$, and let $k$ be the largest non-negative integer such that the disjoint cycle notation of $\sigma$
starts with $(1, 2)(3, 4) \ldots (2k - 1, 2k)$. Then,
\begin{enumerate}[(i)]
\item If the cycle containing $2k + 1$ has at least 3 elements, then $\s$ and $\psi(\s)$ are as follows:
\al{
\sigma &= (1, 2)(3, 4). . . (2k - 1, 2k)(2k + 1, a_1, a_2, . . . , a_j ). . . \\
\psi(\sigma) &= (1)(2, 3)(4, 5) . . . (2k, a_1)(2k + 1, a_2, . . . , a_j ). . .
}
Where if $k = 0$, then $\{1, 2, . . . , 2k\} = \emptyset$ and the fixed point in $\psi(\sigma)$ is $a_1$.
\item Otherwise, $\s$ and $\psi(\s)$ are as follows:
\al{
\sigma = (1, 2)(3, 4). . . (2k - 1, 2k)(2k + 1, a_1)(2k + 2, a_2, . . . , a_j ). . .\\
\psi(\sigma) = (1)(2, 3)(4, 5) . . . (2k, 2k + 1)(2k + 2, a_1, a_2, . . . , a_j ). . .
}
\end{enumerate}

The map $\psi$ is conjugate to the map $\a_n$ by an involution on $\S_n$. To see this, consider a derangement $\s$.
If we ``element-reverse'' the derangement (that is, swap $k$ for $n-k+1$ for all $k$ in the disjoint cycle decomposition of $\s$), then apply $\psi$, and then element-reverse again, we obtain $\a_n(\s)$. So we see that the combinatorial proof in \cite{elizalde2020simple}
can be derived from the combinatorial proof for the identity (\ref{eqn:new}).

The map described in \cite{REMMEL1983371} also operates similarly on $\s$ in the case that $n$ appears in a cycle of length at least $3$. In this case, $\s$ is mapped to the pair $(i, \s \setminus n)$ where $i$ is the position of $n$ in the word $W(\s)$, defined in \cite{REMMEL1983371}. This is similar to $f_n\inv \circ \a_n$ in this case: For $\a_n$, if $n$ appears in a cycle of length at least 3, then $\s(n)$ is removed from the cycle it appears in and becomes fixed. After applying $f_n\inv$, the permutation is sent to $(\s(n), \s \setminus n)$. So for this case, both maps essentially remove $n$ from the cycle decomposition and record where it appeared.


In \cite{GESSEL1993189}, Gessel and Reutenauer show that for a given set $A$ which is properly contained in $[n-1]$, the
number of derangements in $\S_n$ with descent set $A$ is equal to the number of
permutations in $\S_n$ with exactly one fixed point and descent set $A$. 
This result also follows via inclusion-exclusion from bijections given by D\'esarm\'enien and Wachs in \cite{DESARMENIEN1993311}, and a bijective proof can be extracted from this. It would be interesting to find a direct bijection showing this result.


\section{An involution on $\S_n$}\label{involution}
Using 
the map $\a_n$, we can extend to permutations with one fixed point as follows:
First define $\gamma_n: E_{n} \to D_{n+1}$ to be the map which sends to $\s$ to $\s (n+1 \,\,\, m)$, where $m$ is the unique fixed point of $\s$.
Then let
\begin{align*}
    \lambda_n: D_n \cup E_n &\ra D_n \cup E_n \\
    \s &\mapsto \begin{cases}
        \a_n(\s) & \text{ if } \s \in D_n \\
        \a_{n+1}(\gamma_n(\s)) \setminus (n+1) & \text{ if } \s \in E_n.
    \end{cases}
\end{align*}





The restriction of $\l_n$ to $D_n$ is the map $\a_n$, and applying $\l_n$ to a permutation $\s \in E_n$ amounts to changing $\s$ into a derangement by adding $n+1$ as a placeholder, applying $\a_{n+1}$, and then removing $n+1$ from the resulting permutation. If $\sigma = \pi_n$, then $\lambda_n$ sends $\s$ to itself. Rephrasing the combinatorial description of $\a_n$ gives a combinatorial description of $\lambda_n(\s)$ for a general $\s \in D_n \cup E_n$. 

\begin{thm} \label{thm:invol}
Let $\s \in D_n \cup E_n$, with $\s \neq \pi_n$. If $\s \in D_n$, then $\lambda(\s) \in E_n$. If $\s \in E_n$, then $\lambda(\s) \in D_n$. Also, $\l_n(\l_n(\s)) = \s$.
\end{thm}

\begin{examples}[$n=5$] \hfill
\begin{itemize}
\item 
$\lambda_5((12)(345))$ :      The permutation $(12)(345)$ is in $D_5$, so we apply $\a_5$. This goes to $(12)(3)(45)$.

\item
$\lambda_5((12)(3)(45))$:     First we change the permutation to $(12)(36)(45)$, adding the 6 as a placeholder. Then we apply $\a_6$ and remove the 6. The permutation gets sent to $(12)(345)$.

\item
$\lambda_5((1234)(5))$: First we change the permutation to $(1234)(56)$ and then apply $\a_6$ and remove the 6. This goes to $(15)(234)$.

\item $\lambda_5((15)(234))$:  The permutation $(15)(234)$ is in $D_5$, so we apply $\a_5$. This goes to $(1234)(5)$.
\end{itemize}
\end{examples}

Having defined $\lambda_n$, we can now directly describe $\a_n\inv$. Let $\tau \in E_n$, $\tau \neq \pi_n$. Since $\a_n$ is equivalent to $\l_n \mid_{D_n \pm \Pi_n}$, and $\l_n$ is an involution, it follows that $\a_n\inv = \l_n \mid_{E_n \mp \Pi_n}$. So we have that $\a_n\inv(\tau) =\l_n(\tau) $ for $\tau \in E_n$.



\section{Maps of nonderangements} \label{nonderangements}

Let $\overline{D}_n$ denote $\S_n \setminus D_n$, the set of nonderangements of $[n]$. Similarly let $\overline{E}_n$ denote $\S_n \setminus E_n$, the set of permutations of $[n]$ which do not have exactly one fixed point.
Having found the bijection $\a_n: D_n \ra E_n (\pm \Pi_n)$, where the $\pm \Pi_n$ indicates that the element $\pi_n$ is either added to the codomain when $n$ is even, or removed when $n$ is odd, we can also define
\al{
\overline{\a}_n: \overline{D}_n &\ra \overline{E}_n \,\, (\mp \Pi_n) \\
\s &\mapsto \begin{cases}
\s & \text{ if } \s \in \overline{E}_n \\
\a_n\inv(\s) & \text{ if } \s \in E_n,
\end{cases}
}
where the $\mp \Pi_n$ similarly represents the addition or removal of $\pi_n$ from the codomain, depending on the parity of $n$.
We have $\pi_n \in E_n \subseteq \overline{D}_n$ only when $n$ is odd. In this case, $\pi_n$ is fixed and we have $\overline{E}_n \cup \Pi_n$ as the image of $\overline{\a}_n$. If $n$ is even, then $\pi_n \in D_n \subseteq \overline{E}_n$. However, there is nothing in $\overline{D}_n$ that maps to $\pi_n$, so in this case the image of $\overline{\a}_n$ is $\overline{E}_n \setminus \Pi_n$.

Let $\overline{d}_n = | \overline{D}_n| = n! - d_n$. Subtracting the equation (\ref{eqn:oneterm}) from the equation $n! = n(n-1)!$, we have
\begin{align}
    n! - d_n &= n(n-1)! - n d_{n-1} - (-1)^n \nonumber\\
&= n( (n-1)! - d_{n-1}) - (-1)^n \nonumber\\
\overline{d}_n&= n \overline{d}_{n-1} - (-1)^{n}. \label{eqn:nond}
\end{align}

We would like to obtain a bijection showing the 1-term identity (\ref{eqn:nond}).
We have just found $\overline{\a}_n: \overline{D}_n \ra \overline{E}_n \,\, (\mp \Pi_n)$. It remains to find a bijection $\zeta_n: \overline{E}_n \ra [n] \times \overline{D}_{n-1}$. Then composing the maps yields a bijection from $\overline{D}_n$ to $[n] \times \overline{D}_{n-1} \,\, (\mp \Pi_n)$. 

In Section \ref{recurrences}, we defined a bijection $f_n\inv: E_n \ra [n] \times D_{n-1}$. The desired bijection $\zeta_n$ should map $\S_n \setminus E_n$ to $[n] \times \S_{n-1} \setminus ([n] \times D_{n-1})$.
This can be accomplished by subtracting $ f_n\inv$ from a bijection from $\S_n$ to $[n] \times \S_{n-1}$, as in \cite{doyle:lrs}, or equivalently, using the involution principle, as follows.
Let 
\al{
H_n: \S_n &\ra [n] \times \S_{n-1} \\
\s &\mapsto (\s(n), \s \setminus n).
}
This map records where $n$ is in the disjoint cycle decomposition of $\s$, and then removes $n$ from the cycle in which it appears. This is a bijection since if we are given a pair $(m, \s')$, we can recover the preimage by inserting $n$ into the cycle of $\s'$ containing $m$, right before $m$.


To see what $\zeta_n$ turns out to do,
let $\s \in \overline{E}_n$. Applying $H_n$ sends $\s$ to $(\s(n), \s \setminus n)$. If $(\s(n), \s \setminus n) \in [n] \times \overline{D}_{n-1}$, then either $\s$ was not a derangement, or $\s$ was a derangement with $n$ in a cycle of length 2. 
Otherwise $\s \setminus n \in {D}_{n-1}$. This implies either $\s$ was a derangement or $\s$ had only $n$ as a fixed point. However, the latter case cannot have happened since $\s \in \overline{E}_n$. Then $\s$ was a derangement, where $n$ must have appeared in a cycle of length at least 3. In this case we send the pair $(\s(n), \s \setminus n)$ through the map $f_n$. This yields the permutation given by replacing $\s(n)$ with $n$ in the disjoint cycle decomposition and fixing $\s(n)$. So if $\s = (\cdots) (\cdots n \,\,\, \s(n) \,\,\, \s(\s(n)) \cdots)$, we now have $f_n(H_n(\s)) = (\cdots) (\cdots n \,\,\, \s(\s(n)) \cdots)(\s(n))$. This permutation is then sent through $H_n$ again, to obtain the pair $(\s(\s(n)), \s \setminus \{n, \s(n)\} \circ (\s(n))) \in [n] \times \overline{D}_{n-1}$. Collecting these results, we have
\al{
\zeta_n: \overline{E}_n &\ra [n] \times \overline{D}_{n-1} \\
\s &\mapsto \begin{cases}
(\s(n), \s \setminus n) &\text{ if $\s \in \overline{D}_n$ or $\s \in D_n$ with $n$ in a 2-cycle} \\
(\s(\s(n)),\s \setminus \{n, \s(n)\} \circ (\s(n)) ) &\text{ if $\s \in D_n$ with $n$ in a cycle of length $\geq 3$.}
\end{cases}
}
Finally we define 
$Z_n: \overline{D}_n \ra [n] \times \overline{D}_{n-1} \mp \Pi_n$ to be $\zeta_n \circ \overline{\a}_n$. Tracing through possible cases, we can directly write down $Z_n$ as follows:
\al{
Z_n: \overline{D}_n &\ra [n] \times \overline{D}_{n-1} \mp \Pi_n \\
\s &\mapsto \begin{cases}
(\s(n), \s \setminus n) &\text{ if $\s$ has a fixed point $\neq n$} \\
(\a_n\inv(\s)(n), \a_n\inv(\s) \setminus n) &\text{ if $n$ is the unique fixed point of $\s$.}
\end{cases}
}
This description of $Z_n$ is shown in the proof of Theorem \ref{thm:subtr}, found in Appendix \ref{proofs}.

The recurrence (\ref{eqn:nond}) is also proven bijectively in \cite{serginon}, via a map $\theta$. This map takes elements of $\overline{D}_n$ to $\mathcal{M}_n$, which they define as the set of permutations that have a marked fixed point and at least one unmarked fixed point. It is not immediately clear if there is a direct way to translate between the map $\theta$ and $Z_n$. 

Lastly, we note that the map $Z_n$ is a composition of maps found via subtraction, and we could have instead directly subtracted the map $A_n$ defined in Section \ref{definingmaps} from $H_n$. This yields a map $Y_n = H_n \setminus A_n$ equivalent to the composition $Z_n$.

\begin{thm} \label{thm:subtr}
    The map $Y_n$ obtained by subtracting $A_n$ from $H_n$ is equivalent to $Z_n$. That is,
$
Y_n = H_n \setminus A_n = H_n \setminus (f_n\inv \circ \a_n) = (H_n \setminus f_n\inv) \circ \overline{\a}_n = (H_n \setminus f_n\inv) \circ (\id \setminus \a_n) = \zeta_n \circ \overline{\a}_n = Z_n.
$
\end{thm}



\appendixtitleon
\begin{appendices}

\section{Proofs}

\label{proofs}

\begin{L1}
Let $\pi_n \in \S_n$ be as defined earlier.  For all $n \geq 1$, $\a_n$ and $\a_n\inv$ fix $\pi_n$.
\end{L1}

\begin{proof}[Proof of Lemma \ref{lem:pimap}]
We will prove the result by induction. For the base cases $n = 0,1$, we see that each of $\a_0$, $\a_1$, and their inverses map $\pi_n$ to itself by definition.
Now suppose for induction that $\a_{n-1}(\pi_{n-1}) = \pi_{n-1}$ and $\a_{n-1}\inv(\pi_{n-1}) = \pi_{n-1}$.

If $n$ is even, $\a_n$ first sends $\pi_n$ though $\ph_n$, which sends $\pi_n$ to $\pi_n \setminus n$. This is the same permutation as $\pi_{n-1}$, and since $n$ is even, $\pi_{n-1} \in E_n$.
Then this is sent through $\a_{n-1}\inv$. By the inductive hypothesis, $\a_{n-1}\inv$ fixes $\pi_{n-1}$, 
so this is then sent through $\ell_n\inv$, which takes $\pi_{n-1}$ to $\pi_n$. Finally, $\pi_n$ is then sent through the identity, so we obtain $\a_n(\pi_n) = \pi_n$ for $n$ even.


Now suppose $n$ is odd. In this case, $\a_n$ first sends $\pi_n$ through $\ell_n$, to obtain $\pi_{n-1}$. This is then sent through $\a_{n-1}\inv$, which fixes $\pi_{n-1}$ by the inductive hypothesis. Then $\pi_{n-1}$ is mapped via $g_n\inv$ to the pair $(n, \pi_{n-1})$, which is then sent through $f_n$, which just adds back the 1-cycle $n$ to $\pi_{n-1}$, resulting in the permutation $\pi_n$. So $\a_n(\pi_n) = \pi_n$ when $n$ is odd.



Since $\a_n(\pi_n) = \pi_n$ for all $n \geq 1$, it follows that $\a_n\inv(\pi_n) = \pi_n$ as well.
Thus the result holds by induction. 
\end{proof}

\begin{T1}
The combinatorial description of $\a_n$ matches the recursive definition of $\a_n$.
\end{T1}
\begin{proof}[Proof of Theorem \ref{thm:desc}]
Let $n\geq 1$. 
First consider $\pi_n$, which is the product of simple transpositions if $n$ is even, or the product of simple transpositions with the 1-cycle $(n)$ if $n$ is odd. The bijection maps $\pi_n$ to itself always, by Lemma \ref{lem:pimap}. 

Now, let $\sigma \in D_n$, $\sigma \neq \pi$.
We will show the result by induction. 

The smallest base case where $D_n$ or $D_n \cup \Pi_n$ contains an element other than $\pi_n$ is when $n=3$. We have $D_3 \cup \Pi_3 = \{ (123), (132), (12)(3)\}$. Applying $\a_3$ to each of these elements, we have
\al{
(123) \xrightarrow{\ph_n} (1, (12)) \ra (1, (12)) \ra (1, (12)) \ra (1, (12)) \xrightarrow{f_n} (1)(23) = \a_3((123)), \\
(132) \xrightarrow{\ph_n} (2, (12)) \ra (2, (12)) \ra (2, (12)) \ra (2, (12)) \xrightarrow{f_n} (13)(2) = \a_3((132)).
}
And by the lemma, $\a_3((12)(3)) = \a_3(\pi_3) = (12)(3)$, which matches the combinatorial description on $\pi_n$.

To see that the mapping for the other elements of $D_3 \cup \Pi_3$ matches the combinatorial description, for $(123)$, there is no pattern of simple transpositions at the end, so $j = 4$. Then $j-1 = 3$, which appears in a cycle of length greater than 2, so this is in case two of the description, with $a = 1$. Then $\a_n((123)) = (1)(23)$. Similarly, we can check that the combinatorial description maps $(132)$ to $(13)(2)$, so the base case holds.

Now suppose for induction that the combinatorial description matches for $\a_{n-1}$. Then, when sending $\s \in D_n$ through the map $\a_n$, we first apply $\ph_n$. If $\s$ had $n$ in a cycle of length greater than 2, then $\s$ is sent to $(\s(n), \s \setminus n)$, which remains fixed until we apply the map $f_n$. This sends $(\s(n), \s \setminus n)$ to the permutation obtained from $\s \setminus n$ by replacing $\s(n)$ with $n$, and then fixing $\sigma(n)$. This is equivalent to pulling $\sigma(n)$ out of the disjoint cycle decomposition of $\sigma$ and fixing it.

To see that this matches the combinatorial description, if $n$ is in a cycle of length greater than 2, then there is no pattern of simple transpositions at the end of $\s$, so we let $j = n+1$. Then $j-1 = n$, which is in a cycle of length greater than 2, so we apply case two of the description. Then $\s = (\cdots n \,\, \s(n))$ is sent to $(\cdots n \cdots) (\s(n))$, which is the same as the permutation obtained from $\s$ by removing $\s(n)$ from the disjoint cycle decomposition of $\s$ and fixing it.

Otherwise if $n$ appeared in a 2-cycle in $\s$, then upon applying $\ph_n$, $\s$ is sent to $\s \setminus n \in E_n$, which is then sent through $\a_{n-1}\inv$. 

We now check where $\a_{n-1}\inv$ sends this element. First, this sends $\s \setminus n$ through $f_{n-1}\inv$, which gives an ordered pair whose first coordinate is the fixed point $\s(n)$, and for the second coordinate, gives the permutation obtained from $\s \setminus n$ by swapping the fixed point, $\s(n)$, with $n-1$, and then removing $(n-1)$.

Case 1. If $\s(n)$ was $n-1$, then $\s$ was of the form $(\cdots) (n-1 \,\, n)$.
Applying $f_{n-1}\inv$ to $\s \setminus n$ yields the pair $(n-1, \s \setminus \{n, n-1\}) \in [n-1] \times D_{n-2}$. The next step in the definition of $\a_{n-1}\inv$ is to apply $g_{n-1}$ to this, which in this case sends the pair to $ \s \setminus \{n, n-1\} \in D_{n-2}$. Following the next arrow in the definition of $\a_{n-1}\inv$, we send this through $\a_{n-2}$. 

Case 1a. If $\s \setminus \{n, n-1\}$ is of the form $$(\cdots)(\cdots j-2 \,\,\, b \cdots ) (j-1 \,\,\, a) (j \,\,\, j+1) \cdots (n-3\,\,\, n-2),$$ then by the inductive hypothesis, 
$$\a_{n-2}(\s \setminus \{n, n-1\}) = (\cdots)(\cdots j -2 \,\,\, a \,\,\, b \cdots ) (j-1 \,\,\, j) (j+1 \,\,\, j+2) \cdots (n-4\,\,\, n-3)(n-2),$$
with any elements above $n-2$ excluded. Then $\a_{n-1}\inv$ finally sends this through $\ph_{n-1}\inv$, which just adds $n-1$ to the cycle containing the fixed point. So we have 
$$\a_{n-1}\inv(\s \setminus n) = (\cdots)(\cdots j -2 \,\,\, a \,\,\, b \cdots ) (j-1 \,\,\, j) (j+1 \,\,\, j+2) \cdots (n-2 \,\, n-1) \in D_{n-1}$$
in this case. 

Case 1b. If instead $\s \setminus \{n, n-1\}$ was of the form 
$$(\cdots) (\cdots j-1 \,\,\, a \,\,\, b \cdots ) (j \,\,\, j+1) \cdots (n-3\,\,\, n-2),$$ 
then by the inductive hypothesis,
$$\a_{n-2}(\s \setminus \{n, n-1\}) = (\cdots)(\cdots j-1 \,\,\, b \cdots ) (j \,\,\, a) (j+1 \,\,\, j+2) \cdots (n-4\,\,\, n-3)(n-2),$$
with any elements above $n-2$ excluded. Then $\a_{n-1}\inv$ finally sends this through $\ph_{n-1}\inv$, which again just adds $n-1$ to the cycle containing the fixed point. So we have 
$$\a_{n-1}\inv(\s \setminus n ) = (\cdots)(\cdots j-1 \,\,\, b \cdots ) (j \,\,\, a) (j+1 \,\,\, j+2) \cdots (n-2 \,\, n-1)  \in D_{n-1}$$
in this case. 

Next in the definition of $\a_n$, we apply $g_n\inv$, which sends $\a_{n-1}\inv(\s \setminus n )$ to the ordered pair with $n$ as the first coordinate and the same permutation $\a_{n-1}\inv(\s \setminus n)$ as the second coordinate, and then finally this is sent through $f_n$, which just adds the 1-cycle $(n)$ to the permutation $\a_{n-1}\inv(\s \setminus n )$.

So for a permutation $\s$ of the form $(\cdots) (n-1 \,\, n)$, we have that if 

$$\s = (\cdots)(\cdots j-2 \,\,\, b \cdots ) (j-1 \,\,\, a) (j \,\,\, j+1) \cdots (n-1\,\,\, n),$$ then
$$\a_{n}(\s) = (\cdots)(\cdots j -2 \,\,\, a \,\,\, b \cdots ) (j-1 \,\,\, j) (j+1 \,\,\, j+2) \cdots (n-2 \,\, n-1)(n),$$
and if
$$\s = (\cdots) (\cdots j-1 \,\,\, a \,\,\, b \cdots ) (j \,\,\, j+1) \cdots (n-1\,\,\, n),$$ 
then
$$\a_{n}(\s) = (\cdots)(\cdots j-1 \,\,\, b \cdots ) (j \,\,\, a) (j+1 \,\,\, j+2) \cdots (n-2 \,\, n-1)(n).$$


Case 2. If $\s(n) = a \neq n$, then $f_{n-1}\inv$ sends $\s \setminus n = (\cdots)(n-1 \,\, b \cdots) (a)$  to the pair $(a, (\cdots)(a \,\, b \cdots) ) \in [n-1] \times D_{n-2}.$ This is then fixed by $g_n$, fixed by the following arrow in the definition of $\a_{n-1}\inv$, and then sent through $\ph_{n-1}\inv$. This adds $n-1$ into the cycle containing $a$ right before $a$, to yield a derangement $(\cdots)(n-1 \,\, a \,\, b \cdots) \in D_{n-1}$. 

Next in the definition of $\a_n$, we apply $g_n\inv$, which sends this to the ordered pair with $n$ as the first coordinate and the same permutation $(\cdots)(n-1 \,\, a \,\, b \cdots) $ as the second coordinate, and then finally this is sent through $f_n$, which just adds the 1-cycle $(n)$ to the permutation  $(\cdots)(n-1 \,\, a \,\, b \cdots) $. So in this case, $\a_n$ sends a derangement of the form $(\cdots)(n-1 \,\, b \cdots )(a \,\, n)$ with $a \neq n-1$ to $(\cdots)(n-1 \,\, a \,\, b \cdots)(n)$. This matches Case 1 of the combinatorial description in Section \ref{sec:description}.


So in all cases, the recursive definition of $\a_n$ is equivalent to the combinatorial description; thus the result holds by induction.
\end{proof}

\begin{T2}
Let $\s \in D_n \cup E_n$, with $\s \neq \pi_n$. If $\s \in D_n$, then $\lambda(\s) \in E_n$. If $\s \in E_n$, then $\lambda(\s) \in D_n$. Also, $\l_n(\l_n(\s)) = \s$.
\end{T2}

\begin{proof}[Proof of Theorem \ref{thm:invol}]


   We will use the following alternative interpretation of the map $\l_n$: Let $\s \in D_n \cup E_n$. If $\s$ has any fixed point $m$, we add the 2-cycle $(n+1 \,\, m)$. The $n+1$ is a placeholder and will be discarded after applying the map to $\sigma$. Let $N = n$ if $n+1$ was not added as a placeholder, and otherwise let $N = n+1$.
Define $\pi_{N,j}$ to be a permutation on $[N] \setminus [j-1]$ given by

$$\pi_{N,j} = \begin{cases} (j \,\, j+1) \cdots (N-1 \,\, N) &\text{if }N-j+1\text{ is even} \\
(j \,\, j+1) \cdots (N-2 \,\, N-1)(N) &\text{if } N-j+1\text{ is odd.}\end{cases}$$

Then let $j$ be the smallest integer such that the disjoint cycle decomposition of $\s$ is as follows:
$$
\sigma = \delta \circ \pi_{N,j}
$$
with $\delta \in D_{N-1}$, having no copy of $\pi_{j-1,i}$ at the end. If there is no pattern of simple transpositions at the end, we let $j= N+1$. Then we send $\sigma$ through the following map.

Case 1. If $\delta$ has $j-1$ in a 2-cycle, we have 
$$
\sigma = (\cdots) (\cdots j-2 \cdots ) (j-1 \,\, a) \circ \pi_{N,j}
$$
which is sent to
$$
\tau = (\cdots) (\cdots j -2 \,\, a \cdots ) (j-1 \,\, j) \circ \pi_{N,j+1}.
$$

where any values above $n$ are excluded from the cycles.

Case 2. If $\delta$ does not have $j-1$ in a 2-cycle, we have
$$
\sigma = (\cdots)(\cdots j-1 \,\, a \cdots ) \circ \pi_{N,j}
$$
which is sent to
$$
\tau = (\cdots)(\cdots j-1 \cdots ) (j \,\, a) \circ \pi_{N,j+1}.
$$

Again, any values above $n$ are excluded from the cycles. We show that this equivalent map has the desired property. Let $\s \in D_n \cup E_n$, $\s \neq \pi_n$.


Case 1. Suppose $\s \in D_n$. Since $\s$ has no fixed point, we do not add $n+1$ to the disjoint cycle decomposition, and $N = n$.

Case 1a. If $\sigma$ is of the form $(\cdots)(\cdots j-2 \cdots ) (j-1 \,\, a) \circ \pi_{n,j}$ for some minimal $j$, then
$$
\l_n(\s) = (\cdots)(\cdots j -2 \,\, a \cdots ) (j-1 \,\, j) \circ \pi_{n,j+1}
$$
where any values above $n$ are excluded from the cycles. 

Note that if $j$ was equal to $n+1$ when we first applied $\l_n$, then $j-1 = n$ is now fixed by $\l_n(\s)$ since this is in the first case of the definition of $\l_n$. Otherwise if $j$ was less than $n+1$, then the disjoint cycle decomposition of $\s$ ends with a nonempty permutation 
$\pi_{n,j}$, and since $\s \in D_n$, it follows that $n-j+1$ is even. Then $n- (j+1) + 1$ is odd, so $\pi_{n,j+1}$ contains the 1-cycle $(n)$.
Also since $j$ was less than $n+1$, every element up to $j$ appears in a cycle of length 2 or more in $\l_n(\s)$, so the only fixed point of $\l_n(\s)$ occurs in $\pi_{n,j+1}$. Either way, in the first case of the definition of $\l_n$, $n$ becomes fixed by $\l_n(\s)$, so $\l_n(\s) \in E_n$. 

Then, upon applying $\l_n$ again, we add $n+1$ to the cycle $(n)$, so we apply $\l_n$ to the permutation $(\cdots)(\cdots j -2 \,\, a \cdots ) (j-1 \,\, j) \circ \pi_{n+1,j+1}$. Then we
have
\al{
\l_n(\l_n(\s)) &= \l_n((\cdots)(\cdots j -2 \,\, a \cdots ) (j-1 \,\, j) \circ \pi_{n+1,j+1}) \\
&= (\cdots)(\cdots j -2 \cdots ) (j-1 \,\, a) \circ \pi_{n+1,j},
}
which is equal to $\s$ after removing any elements greater than $n$.

Case 1b. Otherwise, $\s$ is of the form $(\cdots)(\cdots j-1 \,\, a \cdots ) \circ \pi_{n,j}$ for some minimal $j$. Then
$$
\l_n(\s) = (\cdots)(\cdots j-1 \cdots ) (j \,\, a) \circ \pi_{n,j+1}.
$$
where any values above $n$ are excluded from the cycles. 

If $j$ was equal to $n+1$ when $\l_n$ was first applied, then the disjoint cycle decomposition of $\s$ had no pattern of simple transpositions 
at the end, 
so $\s = (\cdots)(\cdots j-1 \,\, a \cdots)$. In this case, $a = \s(j-1)$ is now fixed by $\l_n(\s)$, so $\l_n(\s) \in E_n$. Then when applying $\l_n$ again, we are applying $\l_n$ to the permutation 
$$(\cdots)(\cdots j-1 \cdots ) (n+1 \,\, a),$$ which is mapped under the first case of $\l_n$ to $(\cdots)(\cdots j-1 \,\, a \cdots) = \s$.

If $j$ was less than $n+1$, then, as argued previously, $\pi_{n,j+1}$ contains $n$ as a fixed point. Also, every element up to $j$ appears in a cycle of length 2 or more in $\l_n(\s)$, so the only fixed point of $\l_n(\s)$ occurs in $\pi_{n,j+1}$; thus $\l_n(\s) \in E_n$.

Also in this case, when applying $\l_n$ again, we add $n+1$ to the cycle $(n)$, so we apply $\l_n$ to the permutation $(\cdots)(\cdots j -1 \cdots ) (j \,\, a) \circ \pi_{n+1,j+1}$. Then we have
\al{
\l_n(\l_n(\s)) &= \l_n((\cdots)(\cdots j-1 \cdots ) (j \,\, a) \circ \pi_{n+1,j+1}) \\
&= (\cdots)(\cdots j-1 \,\, a \cdots ) (j \,\, j+1) \circ \pi_{n,j+2}\\
&= (\cdots)(\cdots j-1 \,\, a \cdots ) \circ \pi_{n+1,j},
}
which is equal to $\s$ after removing any elements greater than $n$.

Case 2. Suppose $\s \in E_n$. In this case $\s$ has some fixed point $m$, so we add $n+1$ to the cycle containing $m$. After making this modification, we have two cases depending on the disjoint cycle decomposition of $\s$.

Case 2a. If $\sigma$ is of the form $(\cdots)(\cdots j-2 \cdots ) (j-1 \,\, a) \circ \pi_{n+1,j}$ for some minimal $j$, then $\l_n(\s) = (\cdots)(\cdots j -2 \,\, a \cdots ) (j-1 \,\, j) \circ \pi_{n+1,j+1},$
where any values above $n$ are excluded from the cycles, so this is equivalent to
$$\l_n(\s) = (\cdots)(\cdots j -2 \,\, a \cdots ) (j-1 \,\, j) \circ \pi_{n,j+1},$$
where any values above $n$ are excluded from the cycles. Note that if $j$ was equal to $n+2$ when $\l_n$ was first applied, then $j-1 = n+1$ was in a 2-cycle with $a$, which was originally the fixed point of $\s$. Then in $\l_n(\s)$, $a$ is put into a cycle after $j-2 = n$, so it is no longer fixed. Also if $j$ was equal to $n+2$, then $\l_n(\s)$ has no 
pattern of simple transpositions
at the end after removing the elements larger than $n$, so $\l_n(\s) = (\cdots)(\cdots j-2 \,\, a \cdots) = (\cdots)(\cdots n \,\, a \cdots) \in D_n$.

Also, applying $\l_n$ again, we have that 
\al{
\l_n(\l_n(\s)) &= \l_n((\cdots)(\cdots n \,\, a \cdots)) \\
&= (\cdots)(\cdots n \cdots) (n+1 \, a).
}
which is equal to the original $\s$ after removing $n+1$. 

If $j$ was less than $n+2$, then the disjoint cycle decomposition of $\s$ ended with a nonempty permutation $\pi_{n+1,j}$, and since $\s$ was originally in $E_n$ and we added $n+1$ so that $\s$ would have no fixed point, it follows that $n+ 1-j+1$ is even. Then $n+1 - (j+1) + 1$ is odd, so $\pi_{n+1,j+1}$ contains 1 fixed point, which can only be $n+1$.
Also since $j$ was less than $n+2$, every element up to $j$ appears in a cycle of length 2 or more in $\l_n(\s)$, so the only fixed point in $\l_n(\s)$ can be $n+1$, but this will be removed in the end since $n+1 > n$, leaving $\l_n(\s)$ with no remaining fixed points. Thus $\l_n(\s) \in D_n$. 
Applying $\l_n$ again, we have
\al{
\l_n(\l_n(\s)) &= \l_n((\cdots)(\cdots j -2 \,\, a \cdots ) (j-1 \,\, j) \circ \pi_{n,j+1}) \\
&= \l_n((\cdots)(\cdots j -2 \,\, a \cdots ) \circ \pi_{n,j-1}) \\
&= (\cdots)(\cdots j - 2 \cdots) (j-1 \,\, a) \circ \pi_{n, j},
}
which is equal to the original $\s$ after removing elements larger than $n$.

Case 2b. Otherwise, $\s$ is of the form $(\cdots)(\cdots j-1 \,\, a \cdots ) \circ \pi_{n+1,j}$ for some minimal $j$. Then $
\l_n(\s) = (\cdots)(\cdots j-1 \cdots ) (j \,\, a) \circ \pi_{n+1,j+1}$,
where any values above $n$ are excluded from the cycles, so this is equivalent to
$$
\l_n(\s) = (\cdots)(\cdots j-1 \cdots ) (j \,\, a) \circ \pi_{n,j+1},$$
where any values above $n$ are excluded from the cycles.
In this case, $j$ cannot have been equal to $n+2$, because if it were, then the disjoint cycle decomposition of $\s$ had no nonempty permutation $\pi_{n+1,j}$ at the end, so $\s = (\cdots)(\cdots j-1 \,\, a \cdots) = (\cdots)(\cdots n+1 \,\, a \cdots)$. However, we could only add $n+1$ in a transposition with the original fixed point of $\s$, so $n+1$ would not appear in a cycle of length greater than 2. Also, $j$ cannot equal $n+1$ since if it did, then $\s$ would have ended with the 1-cycle $\pi_{n+1, n+1} = (n+1)$, and this cannot happen because again, $n+1$ can only appear in a transposition.

Thus $j$ must have been less than $n+1$, so a nontrivial copy of $\pi_{n+1,j}$ appeared in the disjoint cycle decomposition of $\s$. As argued previously, it follows that $\pi_{n+1, j+1}$ contains 1 fixed point, which can only be $n+1$. Again, since $j$ was less than $n+1$, every element up to $j$  appears in a cycle of length at least 2, so the only fixed point in $\l_n(\s)$ can be $n+1$, which will be removed at the end since $n+1 > n$. Thus $\l_n(\s)$ has no fixed point, so $\l_n(\s) \in D_n$.

Also, when applying $\l_n$ again, we have
\al{
\l_n(\l_n(\s)) &= \l_n((\cdots)(\cdots j-1 \cdots ) (j \,\, a) \circ \pi_{n,j+1}) \\
&= (\cdots)(\cdots j-1 \,\, a \cdots ) (j \,\, j+1) \circ \pi_{n,j+2}\\
&= (\cdots)(\cdots j-1 \,\, a \cdots ) \circ \pi_{n,j},
}
which is equal to the original $\s$ after removing any elements greater than $n$.

So for $\s \neq \pi_n$, if $\s \in D_n$, then $\l_n(\s) \in E_n$, and if $\s \in E_n$, then $\l_n(\s) \in D_n$. Also, in all cases, $\l_n(\l_n(\s)) = \s$; thus $\l_n$ is an involution on $D_n \cup E_n$ which swaps the elements of $D_n$ and $E_n$ excluding $\pi_n$.
\end{proof}

\remove{
\begin{lem}\label{shortlemma}
If $n$ appears in a transposition in $\s \in D_n$, then $\a_n(\s)$ fixes $n$.
\end{lem}

\begin{proof}
    Let $\s \in D_n$ with $n$ appearing in a transposition in $\s$. If $\s(n) = n-1$, then in the combinatorial description of $\a_n$, we have $j<n$. Then in $\a_n(\s)$, $n$ is fixed. Otherwise if $\s(n) = a \neq n-1$, then we have $j = n+1$ and so $n = j-1$ is in a transposition, so this is in Case 1 of the combinatorial description of $\a_n$, which results in $n$ being fixed in $\a_n(\s)$.
\end{proof}
}

\begin{T3}
    The map $Y_n$ obtained by subtracting $A_n$ from $H_n$ is equivalent to $Z_n$. That is,
$
Y_n = H_n \setminus A_n = H_n \setminus (f_n\inv \circ \a_n) = (H_n \setminus f_n\inv) \circ \overline{\a}_n = (H_n \setminus f_n\inv) \circ (\id \setminus \a_n) = \zeta_n \circ \overline{\a}_n = Z_n.
$
\end{T3}
\begin{proof}[Proof of Theorem \ref{thm:subtr}]
The map $Y_n: \overline{D}_n \to [n] \times \overline{D}_{n-1} \mp \Pi_n$ obtained by subtracting $A_n$ from $H_n$ can be found as follows: Let $\s \in \overline{D}_n$. Then apply $H_n$ to obtain $(\s(n), \s \setminus n)$. If $\s \setminus n \in \overline{D}_{n-1}$, then we are done. This occurs when $\s$ had a fixed point other than $n$. Otherwise, if $n$ was the unique fixed point of $\s$, then $\s \setminus n \in D_{n-1}$ so we apply $A_n\inv = \a_n\inv \circ f_n$ to $(\s(n), \s \setminus n)$. 
Since $n$ was the unique fixed point, $(\s(n), \s \setminus n) = (n, \s \setminus n)$, and $f_n$ sends this to $\s \setminus n \circ (n) = \s$. Then we apply $\a_n\inv$ and reapply $H_n$. 
Thus if $n$ was the unique fixed point, $Y_n$ sends $\s$ to $(\a_n\inv(\s)(n), \a_n\inv(\s) \setminus n)$.

This is our final description of $Y_n$:
\al{
Y_n: \overline{D}_n &\ra [n] \times \overline{D}_{n-1} \mp \Pi_n \\
\s &\mapsto \begin{cases}
(\s(n), \s \setminus n) &\text{ if $\s$ has a fixed point $\neq n$} \\
(\a_n\inv(\s)(n), \a_n\inv(\s) \setminus n) &\text{ if $n$ is the unique fixed point of $\s$.}
\end{cases}
}
Now, $Z_n = \zeta_n \circ \overline{\a}_n$ maps as follows: Let $\s \in \nD_n$. Then apply $\overline{\a}_n$. If $\s \in \nE_n$, this does nothing. Then we apply $\zeta_n$ to $\s$. Since $\s \in \nD_n$, it gets sent to $(\s(n), \s \setminus n)$.

Suppose otherwise that $\s \in E_n$. 
In this case, $\overline{\a}_n (\s) = \a_n\inv(\s)$. Let $\t = \a_n\inv(\s) \in D_n$. Then this gets sent through $\zeta_n$. If $\t$ has $n$ in a 2-cycle, then $\zeta_n(\t) = (\t(n), \t \setminus n)$. Otherwise if $n$ appears in $\t$ in a cycle of length 3 or more, $\zeta_n(\tau) = (\t(\t(n)), \t \setminus \{n, \t(n)\} \circ (\t(n))$.

We claim that $n$ appears in a transposition in $\t = \a_n\inv(\s)$ if and only if $n$ was a fixed point of $\s$. Suppose that $n$ appears in a transposition in $\t$.  To recover $\s$, we apply $\a_n$ to $\t$. If $\t(n) = n-1$, then in the combinatorial description of $\a_n$, we have $j<n$. Then in $\s = \a_n(\t)$, $n$ is fixed. Otherwise if $\t(n) = a \neq n-1$, then we have $j = n+1$ and so $n = j-1$ is in a transposition, so this is in Case 1 of the combinatorial description of $\a_n$, which results in $n$ being fixed in $\a_n(\t)$.

On the other hand, suppose $n$ was fixed in $\s$. We can check where $n$ appears in $\a_n\inv(\s) = \l_n(\s)$. Since $n$ was fixed, we have $n+1$ in a transposition in $\gamma_n(\s)$. Then applying $\a_{n+1}$ to this, we have $j<n+1$, so $n+1$ becomes fixed and $n$ appears in a transposition in $\a_n\inv(\s)$.


So we have that $\t$ has $n$ in a 2-cycle if and only if $n$ is fixed in $\s$, and $\s \in E_n$, so $n$ is the unique fixed point of $\s$. Thus $\zeta_n(\a_n\inv(\s)) = (\t(n), \t \setminus n)$ when $\s$ had $n$ fixed.

It follows from the claim that if $n$ appears in a cycle of length 3 or more in $\t$, then $n$ was not fixed in $\s$. In this case, $\s$ has one of the following cycle decompositions:

Case 1. $\s = (\cdots) (m)(n \, a)$, for some $m, a \neq n$. In this case, $\t = \a_n\inv(\s) = (\cdots)(n\, m\, a)$.

Case 2. $\s = (\cdots) (m)(n \, a\, b \, \cdots)$, for some $m, a,b \neq n$. In this case, $\t = \a_n\inv(\s) = (\cdots)(n\, m\, a\, b \, \cdots)$.

In both cases, the original fixed point $m$ is put into the cycle containing $n$ just after $n$ and just before $\s(n)$. So $m = \t(n)$, and $\t(\t(n)) = \s(n)$. 

Also, since $\t = \a\inv(\s)$ is obtained by adding $m$ to the cycle containing $n$ after $n$, it follows that $\t \setminus \{ n, \t(n)\} \circ (\t(n)) = \s \setminus n$, since removing $n$ and $m$ and then fixing $m$ in $\t$ is the same as simply removing $n$ from $\s$.

So if $n$ was not the unique fixed point of $\s \in E_n$, we have that $\zeta_n(\a_n\inv(\s)) = (\s(n), \s \setminus n)$.


Collecting these results, we have the following description of $Z_n$:
\al{
Z_n: \overline{D}_n &\ra [n] \times \overline{D}_{n-1} \mp \Pi_n \\
\s &\mapsto \begin{cases}
(\s(n), \s \setminus n) &\text{ if $\s \in \nE_n$ or if $\s$ has a unique fixed point $m \neq n$} \\
(\t(n), \t \setminus n) &\text{ if $n$ is the unique fixed point of $\s$.}
\end{cases}
}

This is the same as 
\al{
Z_n: \overline{D}_n &\ra [n] \times \overline{D}_{n-1} \mp \Pi_n \\
\s &\mapsto \begin{cases}
(\s(n), \s \setminus n) &\text{ if $\s$ has a fixed point $\neq n$} \\
(\a_n\inv(\s)(n), \a_n\inv(\s) \setminus n) &\text{ if $n$ is the unique fixed point of $\s$.}
\end{cases}
}
which is exactly the description of $Y_n$.
\end{proof}

\end{appendices}
\section*{Acknowledgements}
{I would like to thank my advisor Peter Doyle for many valuable insights on this work.} I would also like to thank Yan Zhuang for helpful 
correspondence.

\printbibliography

\end{document}